\documentclass{amsart}
\usepackage{amssymb}
\usepackage{amsthm}
\usepackage{mathtools}
\usepackage{environ}
\usepackage{enumerate}
\usepackage{microtype}
\usepackage{enumitem}
\usepackage{mathrsfs}
\usepackage[utf8]{inputenc}
\usepackage{amsaddr}
\usepackage{orcidlink}
\usepackage{tikz}
\usetikzlibrary{patterns}
\usepackage{caption}

\theoremstyle{plain}
\newtheorem{theorem}{Theorem}[section]
\newtheorem{lemma}[theorem]{Lemma}
\newtheorem{proposition}[theorem]{Proposition}
\newtheorem{corollary}[theorem]{Corollary}
\newtheorem*{theoremA}{Theorem A}
\newtheorem*{theoremB}{Theorem B}
\newtheorem*{theoremC}{Theorem C}

\newtheorem*{theorem*}{Theorem}

\theoremstyle{definition}
\newtheorem{definition}[theorem]{Definition}
\newtheorem{example}[theorem]{Example}

\theoremstyle{remark}
\newtheorem*{remark}{Remark}

\DeclareFontFamily{U}{mathx}{\hyphenchar\font45}
\DeclareFontShape{U}{mathx}{m}{n}{<-> mathx10}{}
\DeclareSymbolFont{mathx}{U}{mathx}{m}{n}
\DeclareMathAccent{\widebar}{0}{mathx}{"73}

\newcommand{\C}{\mathbb{C}}
\newcommand{\D}{\mathbb{D}}
\newcommand{\R}{\mathbb{R}}
\newcommand{\N}{\mathbb{N}}
\newcommand{\psh}{\mathcal{PSH}}
\newcommand{\usc}{\mathcal{USC}}

\newcommand{\suchthat}{\mathrel{;}}
\DeclarePairedDelimiter\abs{\lvert}{\rvert}

\title[On the size of boundary pluripolar sets]{On the size of boundary pluripolar sets}
\author{Mårten Nilsson\,\orcidlink{0009-0000-4441-4254}}
\thanks{This work was supported by a research grant from the Sverker Lerheden Foundation.}
\address{Department of Mathematics \\
Stockholm University \\
106 91 Stockholm, Sweden}
\email{marten.nilsson@math.su.se}


\begin{document}

\begin{abstract}
We prove a number of results related to the size and propagation of boundary pluripolar sets, the exceptional sets for the Dirichlet problem for the complex Monge--Ampère equation. We extend Stout's result that peak sets on strictly pseudoconvex domains $\Omega\subset\C^N$ must have topological dimension less than $N$ to also encompass non-propagating boundary pluripolar $F_\sigma$ sets. In particular, boundary pluripolar sets must propagate into the interior if their topological dimension exceeds $N-1$. We also prove that sets of sufficiently small Hausdorff dimension must be boundary pluripolar and non-propagating, provided that the domain admits peak functions with sufficient boundary regularity. Lastly, we prove that the class of Jensen measures and the class of representing measures do not coincide on any smooth, strictly pseudoconvex domain. This extends a result of Hedenmalm. 
\end{abstract}
\maketitle
\section{Introduction}
Let $\Omega$ denote a bounded domain in $\C^N$, and let $A \subset \partial\Omega$ be a set for which there exists a negative plurisubharmonic (psh) function $u$ defined on $\Omega$ ($u\in \psh^-(\Omega)$), not identically $-\infty$, with
\[\limsup_{z\rightarrow A} u(z) = -\infty.\]
Such \textit{boundary pluripolar sets} (b-pluripolar for short) were recently characterized as the exceptional sets for the Dirichlet problem for the complex Monge--Ampère operator ~\cite{nilsson2}. More precisely, under the assumption that $\Omega$ admits strong plurisubharmonic barriers, in the sense that we at every boundary point $\zeta \in \partial\Omega$ can find $u \in \psh(\Omega)\cap C(\widebar \Omega)$ with
\[
u(\zeta)=0, \quad u<0 \text{ on } \widebar \Omega \setminus \{\zeta\},
\]
then the Dirichlet problem 
\[
    \begin{cases}
    u \in L^\infty\cap\psh(\Omega) \\
    \det(\frac{\partial^2 u}{\partial z_j \partial \bar z_k})=0 \\
        \lim_{z \rightarrow \zeta}u(z) = \varphi(\zeta) \qquad \forall \zeta \in \partial\Omega \setminus E_\varphi 
    \end{cases}
\]
 is uniquely solvable (in the weak sense of Bedford--Taylor~\cite{bedford-taylor, bedford-taylor2}) if and only if $E_\varphi$ is b-pluripolar. Here, we assume that $\varphi\colon \partial \Omega \rightarrow \R$ is bounded, and continuous outside $E_\varphi$. Furthermore, the unique solution is continuous~\cite[Theorem~5.1]{nilsson} outside the closure of
    \[
    \hat E_\varphi \coloneqq\{z \in \widebar\Omega \suchthat u\in\psh^-(\Omega), \limsup_{w\rightarrow E_\varphi}u(w)=-\infty \implies u(z) =-\infty\},
    \]
the \textit{boundary pluripolar hull} of $E_\varphi$ (if the boundary pluripolar hull contains interior points, we say that the set is \textit{propagating}). Since discontinuity sets are necessarily $F_\sigma$, this motivates investigating necessary and sufficient criteria for when boundary pluripolar $F_\sigma$ sets propagate. 

In the first part of the paper, we obtain the following result:
\begin{theoremA}
Let $\Omega \subset \C^N$ be a strictly pseudoconvex domain. Then every b-pluripolar $F_\sigma$ set of topological dimension $N$ must propagate into the interior.
\end{theoremA}
\noindent
Contrapositively, non-propagating b-pluripolar sets have at most dimension $N-1$. The proof is essentially an adaptation of an argument due to Stout~\cite{stout}, who proved the corresponding statement for peak sets, i.e. closed sets $K \subset \partial \Omega$ for which there exists $h_K\in A(\Omega)$ (holomorphic on $\Omega$ and continuous on the closure) with
\[
h_K = 1 \text{ on }K, \quad |h_K| < 1 \text{ on }\widebar \Omega \setminus K.
\] 
Roughly, the idea is to show that even under the weaker assumption that $K$ is non-propagating and b-pluripolar, $K$ is still polynomially convex. The final step in Stout's proof 
then implies that $K$ has some topological properties which are incompatible with $\dim K \geq N$.

In Section~3, we establish sufficient criteria for b-pluripolarity in terms of Hausdorff dimension. Given a bounded domain $\Omega$ and a parameter $\beta>0$, we consider the set $P_\beta$ consisting of $\zeta \in \partial \Omega$ for which there exists a peak function $h_{\{\zeta\}}$ with
\[
\abs{h_{\{\zeta\}}(z)-1} \leq C_\zeta \|z-\zeta\|^\beta,
\]
where $\| \cdot\|$ denotes the Euclidean norm. We show 
\begin{theoremB}
Every set $A\subset P_\beta$ of zero $\beta$-dimensional Hausdorff measure is non-propagating and b-pluripolar.
\end{theoremB}
\noindent
 The argument is quite elementary; we simply postcompose peak functions with the harmonic measure of arcs on the circle. 

Finally, in Section~4, we compare representing measures and Jensen measures at $z_0 \in \Omega$ supported on $\partial \Omega$, where $\Omega$ is smooth and strictly pseudoconvex. We show that the compact totally null sets of the two classes differ; the former are the zero sets for $A(\Omega)$, in particular b-pluripolar and non-propagating, while the class of totally null sets for Jensen measures contains every b-pluripolar set that does not propagate to $z_0$. From this, we deduce 
\begin{theoremC}
Let $\Omega\subset \C^N$, $N>1$ be a smooth, strictly pseudoconvex domain. Then there are always representing measures that are not Jensen measures.
\end{theoremC}
\par \noindent
 This extends a result of Hedenmalm (\cite[Theorem~3.7]{hedenmalm}, see also Theorem~\ref{hedenmalmsats} below) on the unit ball.

We have labeled the main results as Theorems A--C for ease of reference;
these correspond to Theorems~\ref{topdim}, \ref{hausdorff}, and Corollary~\ref{hedenmalmgen}.

\section{Boundary pluripolar sets of dimension \texorpdfstring{$N$}{N} must propagate}
Our aim in this section is to prove that on strictly pseudoconvex domains in $\C^N$, 
a non-propagating b-pluripolar $F_\sigma$ set $A$ with $\dim A \geq N$ cannot exist. The argument will be based on the following result, which was established using a cohomological argument in the proof of the main theorem in~\cite{stout}.
\begin{lemma}[Stout]\label{stout}
Let $\Omega\subset \C^N$ be a bounded convex domain, and suppose that $A \subset \partial \Omega$ is a compact set such that for every neighborhood $U$ (in $\partial \Omega$) containing $p\in A$, there is a real hyperplane $\Pi_U: L(z)=0$ such that 
\[
p \in \{L(z)>0\}\cap \partial \Omega \subset U.
\]
If each $(A\cap \{L(z)>0\}) \cup (\Pi_U \cap \widebar \Omega)$ is polynomially convex, then $\dim A < N$. 
\end{lemma}
\begin{figure}[h]
\centering
\tikzset{every picture/.style={line width=0.75pt}}

\begin{tikzpicture}[x=0.75pt,y=0.75pt,yscale=-0.7,xscale=0.7]

\fill[pattern=north west lines,pattern color=black!35]
  (58.57,247.01)
    .. controls (73.86,143.18) and (191.74,62.02) .. (334.27,61.73)
    .. controls (478.52,61.42) and (596.29,144.05) .. (607.65,249.64)
  -- (607.65,247.01)
    .. controls (607.65,258.06) and (484.74,267.01) .. (333.11,267.01)
    .. controls (181.49,267.01) and (58.57,258.06) .. (58.57,247.01)
  -- cycle;

\draw
  (58.57,247.01)
    .. controls (73.86,143.18) and (191.74,62.02) .. (334.27,61.73)
    .. controls (478.52,61.42) and (596.29,144.05) .. (607.65,249.64);

\draw[dash pattern={on 4.5pt off 4.5pt}]
  (58.57,247.01) .. controls (58.57,235.97) and (181.49,227.01) .. (333.11,227.01)
  .. controls (484.74,227.01) and (607.65,235.97) .. (607.65,247.01)
  .. controls (607.65,258.06) and (484.74,267.01) .. (333.11,267.01)
  .. controls (181.49,267.01) and (58.57,258.06) .. (58.57,247.01) -- cycle;

\fill[white]
  (258,77) .. controls (278,67) and (368,57) .. (348,77)
  .. controls (328,97) and (363,114) .. (348,137)
  .. controls (333,160) and (278,167) .. (258,137)
  .. controls (238,107) and (238,87) .. (258,77) -- cycle;

\draw
  (258,77) .. controls (278,67) and (368,57) .. (348,77)
  .. controls (328,97) and (363,114) .. (348,137)
  .. controls (333,160) and (278,167) .. (258,137)
  .. controls (238,107) and (238,87) .. (258,77) -- cycle;

\node at (310,130) {$U$};

\fill[pattern=north east lines,pattern color=black!45]
  (191.56,69.69) -- (618.73,29.09) -- (439.44,86.31) -- (12.27,126.91) -- cycle;
\draw
  (191.56,69.69) -- (618.73,29.09) -- (439.44,86.31) -- (12.27,126.91) -- cycle;

\fill[white]
  (261.56,85.78) .. controls (260.97,80.36) and (276.06,74.26) .. (295.28,72.14)
  .. controls (314.49,70.03) and (330.55,72.71) .. (331.14,78.12)
  .. controls (331.74,83.54) and (316.65,89.64) .. (297.43,91.76)
  .. controls (278.22,93.87) and (262.16,91.2) .. (261.56,85.78) -- cycle;

\draw[dash pattern={on 4.5pt off 4.5pt}]
  (261.56,85.78) .. controls (260.97,80.36) and (276.06,74.26) .. (295.28,72.14)
  .. controls (314.49,70.03) and (330.55,72.71) .. (331.14,78.12)
  .. controls (331.74,83.54) and (316.65,89.64) .. (297.43,91.76)
  .. controls (278.22,93.87) and (262.16,91.2) .. (261.56,85.78) -- cycle;

\fill (290,82) circle (1.6pt) node[right] {$p$};
\end{tikzpicture}

\caption*{\textbf{Figure.} A point $p\in \{L(z)>0\}\cap \partial \Omega \subset U$.} 
\end{figure}
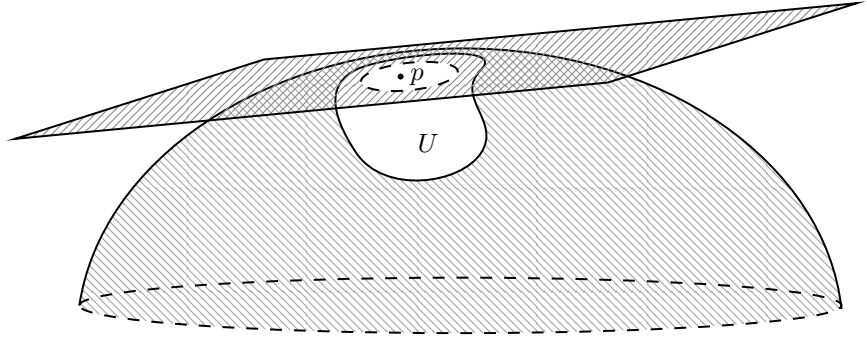

In order to apply this lemma, we begin by proving
\begin{lemma}\label{approx}
Assume that $\Omega$ is convex. Then every $u\in \psh(\Omega)\cap C(\widebar \Omega)$ can be uniformly approximated on $\widebar \Omega$ by elements in $\psh(\C^N)$.
\end{lemma}
\begin{proof}
Since $\Omega$ is convex, the boundary is Lipschitz and we may use the main result of~\cite{hed} to find $\Omega' \Supset \Omega$ and $u'\in \psh(\Omega')\cap C(\Omega')$ with $|u-u'|< \varepsilon$ on $\Omega$. Furthermore, taking a Minkowski sum of $\Omega$ and a small enough ball, we may assume that $\Omega'$ is convex. A result of Bremermann and Sibony (\cite[Théorème~9]{sibony}, see also~\cite[Theorem~1.3.9]{stoutbok}) then yields holomorphic functions $f_1,\dots,f_n\in \mathcal{O}(\Omega')$ and constants $1\geq c_1,\dots, c_n>0$ such that 
\[
\big|u - \max_j (c_j\log|f_j|)\big|<\varepsilon'\quad \text{ on $\Omega$.}
\]
Using that convex sets are polynomially convex, the Oka--Weil theorem lets us replace $f_j$ by holomorphic polynomials $P_j$. This yields approximants defined on $\C^N$.
\end{proof}
We will also need the following. Recall that $\Omega$ is said to be \textit{B-regular} if every boundary point admits a strong plurisubharmonic barrier.
\begin{lemma}\label{polyconv}
Let $\Omega$ be convex and B-regular, $\Pi: L(z)=0$ a real hyperplane, and let $A\subset \partial \Omega$ be a compact b-pluripolar set. If $A$ does not propagate, then $A\cup (\Pi\cap \widebar \Omega)$ is polynomially convex.
\end{lemma}
\begin{proof}
We first show that $A$ is polynomially convex. By \cite[Theorem~4.3.4]{hormander}, it is enough to find for each $z_0\in \C^N\setminus A$ a psh function $u_{z_0}\in\psh(\C^N)$ such that
\[
u_{z_0}(z_0) \geq \max_{z\in A} u_{z_0}(z).
\]
For $z_0 \in \C^N \setminus \widebar \Omega$ we may construct $u_{z_0}$ from the plane separating $z_0$ and $\widebar \Omega$, so assume that $z_0 \in \widebar \Omega\setminus A$. Then, since $A$ is b-pluripolar and non-propagating, there exists $u\in \psh^-(\Omega)$ with 
\[
\limsup_{z\rightarrow A} u(z) =-\infty, \quad u(z_0) \neq -\infty,
\]
and by scaling, we may assume that $u(z_0)>-1$. By Wikström's approximation theorem on B-regular domains~\cite[Theorem~4.1]{wikstrom}, there exists $u_n \in C(\widebar \Omega) \cap \psh(\Omega)$ with $u_n \searrow u$, such that for every $C<-1$, at each $\zeta \in A$, there is a $n_\zeta$ with $u_{n_\zeta}(\zeta)<C$. This yields a covering 
\[
A \subset \bigcup \{u_{n_\zeta}<C\},
\]
and so by compactness and monotonicity (take a finite subcover and the maximal index), there exists $u_{n_0}$ with 
\[
u_{n_0}(z_0)>-1, \quad u_{n_0}<C \text { on }A. 
\]
The claim then follows from Lemma~\ref{approx}. 

To show that $A\cup(\Pi\cap\widebar \Omega)$ is polynomially convex, assume that $z_0 \in \widebar \Omega \setminus(A \cup \Pi)$. Multiplying $L$ with a constant if necessary, we may assume that $L(z_0)>1$. The preceding paragraph implies that we may find $\tilde u \in \psh(\C^N)$ with
\[
\tilde u(z_0)>-1, \quad \tilde u <0 \text { on }\widebar \Omega,\quad \tilde u< -\max_{\zeta\in A}|L(\zeta)| \text { on }A.  
\]
The psh function $\tilde u + L$ then satisfies 
\[
\tilde u(z_0)+L(z_0)>0, \quad \tilde u+L< 0 \text { on }A \cup (\Pi\cap \widebar \Omega),  
\]
which finishes the proof.
\end{proof}
\begin{remark}
The assumption that $\Omega$ is B-regular (which is automatic if $\Omega$ is strictly pseudoconvex) compensates for the fact that $A$ is a peak set in Stout's proof. Indeed, if $A$ is a peak set, we can instead apply Lemma~\ref{approx} to $u = \max(C_1\log|1-h_A|,C_2)$ with suitable constants $C_2 <0<C_1$. 
\end{remark}
The following simple example shows that the converse is not true: There are polynomially convex, b-pluripolar sets which do propagate into the interior.
\begin{example}
Let $\mathbb{B}$ denote the ball in $\C^2$, and take an arc $\gamma \subset \{|z_2|=1\}$. This set is b-pluripolar since $\gamma \subset \{\log|z_1| =-\infty\}$, and polynomially convex by~\cite[Corollary~3.1.2]{stoutbok}.  Let $u \in \psh^-(\mathbb{B})$ be a function such that $\limsup_{w\rightarrow \gamma}u(w)=-\infty$. Since $\gamma$ has positive arc length, it is not b-(pluri)polar on the disk $\{(0, z_2) \suchthat |z_2|<1\}$. It follows that $u(0,z_2) \equiv -\infty$; in other words, $\gamma$ propagates along $\{z_1=0\}\cap\mathbb B$.
\end{example}

With Lemma~\ref{polyconv} at hand, we are now ready to prove the main result in this section. 
\begin{theorem}\label{topdim}
Let $\Omega\Subset \C^N$ be strictly pseudoconvex, and let $A\subset \partial \Omega$ be a boundary pluripolar $F_\sigma$ set with topological dimension greater than $N-1$. Then $\hat A \cap \Omega$ is non-empty.
\end{theorem}
\begin{proof}
Let us first do some standard reductions. If we for a contradiction assume the existence of a non-propagating b-pluripolar $F_\sigma$ set $A$ with $\dim A \geq N$, then 
\[
A= \bigcup K_j
\]
with $K_j$ compact. This implies~\cite[Theorem~3.1.8]{Engelking} that there exists $j_0$ with $\dim K_{j_0}\geq N$. Pick a point $p\in K_{j_0}$ such that each neighborhood $U$ containing $p$ satisfies $\dim U\cap K_{j_0}\geq N$. Since $\Omega$ is strictly pseudoconvex,  there exists an open set $V$ containing $p$ and a local biholomorphism $f$ such that $f(\partial \Omega\cap V)$ is strictly convex in a neighborhood of $K\coloneq f(U\cap K_{j_0})$. Also, if $A$ does not propagate, then $K$ does not propagate into $B(p,\epsilon)\cap f(\Omega\cap V)$. We may therefore assume that the b-pluripolar set $A$ is compact, and that $\Omega$ is convex and B-regular, and even strictly convex in a neighborhood of $A$. 

However, strict convexity in a neighborhood of $p\in A$ implies that the open sets obtained by cutting off $\widebar \Omega$ with hyperplanes $\Pi: L(z)=0$ form a neighborhood basis for $p$. By Lemma~\ref{polyconv}, each
\[
(A\cap\{L(z)>0\}) \cup (\Pi \cap \widebar \Omega)
\]
is polynomially convex, so Lemma~\ref{stout} implies that $\dim A \geq N$ is not possible.
\end{proof}
\begin{remark}
    The condition $\dim A \geq N$ is sufficient but not necessary. This can be seen by translating and scaling the domain so that the function
\[
u(z_1,\dots, z_N) =\log\big(\sum_{j=2}^N|z_j|^2\big)
\]
satisfies $u\mid_\Omega \leq 0$ and $\{u=-\infty\} \cap \Omega \neq \emptyset$. A standard argument using the map $z\mapsto (z, 0, \dots, 0)$ then implies that $\{u=-\infty\}\cap \partial \Omega$ propagates to $\{u=-\infty\}\cap \widebar \Omega$.  Hence, there exist propagating b-pluripolar sets of dimension 1 on every bounded domain in $\C^N$, $N>1$.
\end{remark}

We end this section with the  following observation, which also follows from combining the approximation theorems of Wikström, Oka--Weil and Bremermann--Sibony.
\begin{proposition}
Let $\mathbb{B}\subset \C^N$ denote the unit ball. Then for every $A\subset \partial \mathbb{B}$, $z_0 \in \mathbb{B} \setminus \hat A$, there exist holomorphic polynomials $P_j$ and constants $c_j>0$ such that $\sum c_j \log|P_j|\in \psh^-(\mathbb{B})$, with
\[
 \sum_{j=1}^\infty c_j \log|P_j(z_0)|>-\infty, \quad \limsup_{w\rightarrow A}\sum_{j=1}^\infty c_j \log|P_j(w)|=-\infty.
\]
\end{proposition}
\begin{proof}
If $z_0\in\mathbb{B}\setminus \hat A$, then by definition we can find $u\in \psh^-(\mathbb{B})$ such that $u(z_0)>-\infty$, $\limsup_{w\rightarrow A}u(w)=-\infty$. Reasoning as in the proof of Lemma~\ref{polyconv}, we can from $u$ construct continuous psh functions $u_j$ with
\[
u_j(z_0)>-1/2^j, \quad \limsup_{w\rightarrow A}u_j(w)<-1.
\]
By the proof of Lemma~\ref{approx}, $u_j$ can be uniformly approximated by functions of the form $\max_{1\leq k\leq n} c_{j,k} \log|P_{j,k}|$; in particular we can find a polynomial $P_j$ and $c_j>0$ such that
\[
c_j\log|P_j|<0, \quad c_j\log|P_j(z_0)|>-1/2^j, \quad \limsup_{w\rightarrow A}c_j\log|P_j(w)|<-1.
\]
As $\sum c_j\log|P_j|$ is a decreasing series of psh functions, finite at $z_0$, it is psh and not identically $-\infty$.
\end{proof}

\section{Sets of small Hausdorff dimension are boundary pluripolar}

In this section, our goal is to demonstrate that sets of sufficiently small Hausdorff dimension must be b-pluripolar and non-propagating. To fix notation, given a set $A$, we let $\mathcal{H^\beta}(A)$ denote its $\beta$-dimensional Hausdorff measure, and $\dim_\mathcal{H} A$ its Hausdorff dimension.
 
We begin by extracting the following result from Davie--Øksendal's paper~\cite{oksendal} on peak sets. 
\begin{theorem}\label{oksendalsats}
Let $\Omega\subset \C^N$ be a pseudoconvex domain, and let $P$ denote the set of points $\zeta \in \partial \Omega$ for which we can find a neighborhood $N$ and a $C^2$, strictly psh function $\rho \in \psh(N)$ such that $\rho(\zeta)=0$ and $\rho\leq 0$ in $N\cap \Omega$. Then every $A \subset P$ of zero 1-dimensional Hausdorff measure is non-propagating and b-pluripolar.
\end{theorem}
\begin{proof}
Let $H^\infty(\Omega)$ denote the bounded holomorphic functions on $\Omega$. By \cite[Theorem~2]{oksendal}, for every $K \Subset \Omega$, there exists $h_A \in H^\infty(\Omega)\cap C(\Omega\cup A)$ such that
\[
h_A =1 \text{ on }A, \quad |h_A| < 1 \text{ on }K.
\]
In particular, 
\[
u\coloneq\log|1-h_A|-\sup_{\Omega}\log|1-h_A|
\]
is a negative psh function with $\limsup_{w\rightarrow A}u(w)=-\infty$ and $\sup_{w\in K}u(w)>-\infty$, so $A$ is a b-pluripolar set which does not propagate to $K$. Since $K$ was arbitrary, we conclude that $A$ is non-propagating.
\end{proof}
\begin{remark}
 The condition $\mathcal{H}^1(A)=0$ is the best one can hope for, as the remark to Theorem~\ref{topdim} shows. On the other hand, there exist non-propagating and b-pluripolar sets with ``full'' Hausdorff dimension $2N-1$, as there already exist peak sets with this property~\cite{kot,henriksen}. We also remark that the other main result in Davie--Øksendal~\cite[Theorem~1]{oksendal} allows for an anisotropic version of the above theorem. Roughly, if one additionally assumes that $N \cap \Omega$ is strictly pseudoconvex, then one can allow for sets that have zero 2-dimensional measure in the complex tangent directions, and zero 1-dimensional measure in the orthogonal direction (spanned at every point by $i$ times the normal vector). In particular, smooth arcs for which the tangent at every point lies in the complex tangent space are b-pluripolar.
\end{remark}
We shall suggest another way to construct b-pluripolar sets from peak functions, lifting the above pluripotential content of the Davie--Øksendal theorem to every domain where one can establish a local Hölder estimate for peak functions. More precisely, given a bounded domain $\Omega$, consider
the set $P_\beta \subset \partial \Omega$ consisting of $\zeta \in \partial \Omega$ for which there exists a peak function $h_{\{\zeta\}}\in H^\infty(\Omega)$ with
\[
\abs{h_{\{\zeta\}}(z)-1} \leq C_\zeta \|z-\zeta\|^\beta.
\]
Such peak functions can for example be constructed on various pseudoconvex domains of finite type, see~\cite{noell} for more details. We shall show 
\begin{theorem}\label{hausdorff}
   Every set $A\subset P_\beta$ of zero $\beta$-dimensional Hausdorff measure is non-propagating and b-pluripolar. 
\end{theorem}
 Before proceeding with a proof, we establish the following \textit{exercise} in potential theory in the complex plane.
\begin{lemma}\label{basic}
The harmonic measure $\omega(z,\D, (e^{-i\theta},e^{i\theta}))$ on the arc $(e^{-i\theta},e^{i\theta})$ satisfies
\[
\omega(z,\D, (e^{-i\theta},e^{i\theta})) \geq \frac{1}{2}-\frac{\theta}{2\pi} \quad z\in\D\cap\D(1, |1-e^{i\theta}|).
\]
In particular, 
\[
-3 \cdot \omega(z,\D, \partial\D\cap\D(1,r)) \leq -1 \quad z\in\D\cap\D(1,r)
\]
for $r$ small enough.
\end{lemma}
\begin{proof}
Write $\omega(z,\D, (e^{-i\theta},e^{i\theta})) = \frac{1}{\pi}(\varphi(z) -\theta),$ where $\varphi$ is the angle subtended at $z$ by the arc, and pick $\zeta \in \D\cap\partial\D(1, |1-e^{i\theta}|)$. Applying the inscribed angle theorem twice, we conclude that $\varphi(\zeta) = \frac{\pi+\theta}{2}$. Hence
\[
\omega(\zeta,\D, (e^{-i\theta},e^{i\theta})) = \frac{1}{2}-\frac{\theta}{2\pi}.
\]
The maximum principle then implies the sought inequality.
\end{proof}
\begin{proof}[Proof of Theorem~\ref{hausdorff}] Fix $a \in \Omega$. By assumption, we can for every $\zeta \in A$ pick a peak function $h_{\{\zeta\}}$ and a constant $C_\zeta >0$ such that 
\[
\abs{h_{\{\zeta\}}(z)-1} \leq C_\zeta \|z-\zeta\|^\beta,
\]
and postcomposing with an automorphism $\overline \D \to \overline \D$ (keeping the exponent intact), we may assume that $h_{\{\zeta\}}(a)=0$. Now decompose $A=\bigcup_n A_n$, where
\[
A_n \coloneq\{\zeta \in A \suchthat 2^{n-1}< C_\zeta \leq 2^n\}.
\]
Using $\mathcal{H}^\beta(A) =0$, for every $\varepsilon>\varepsilon_n>0$ such that $\sum \varepsilon_n <\varepsilon$, there is a covering $A_n\subset \bigcup_{j} \mathbb{B}(\zeta_{j,n},r_{j,n})$ with
    \[
    \sum_{j=1}^{\infty} r_{j,n}^\beta <\frac{\varepsilon_n}{2^{n}}.
    \]
Using Lemma~\ref{basic}, we can choose $\varepsilon_n$ small enough so that 
\[
-3\cdot\omega(z,\D, \partial\D\cap\D(1,\varepsilon_n)) \leq -1, \quad z\in\D\cap\D(1,\varepsilon_n).
\]
Define the pluriharmonic functions
\[
u_{j,n}(z) \coloneq - \omega(h_{\{\zeta_{j,n}\}}(z), \partial\D \cap \D(1, 2^n r_{j,n}^\beta)),
\]
and note, since $|\partial\D \cap\D(1, r)|<2\pi r$, that
\begin{align*}
u_{j,n}(a) &= - \omega(0, \D, \partial\D \cap\D(1, 2^n r_{j,n}^\beta)) \\ &=- \frac{|\partial\D \cap\D(1, 2^n r_{j,n}^\beta)|}{2\pi}> - 2^n r_{j,n}^\beta.
\end{align*}
On the other hand, for $z$ with $\|z-\zeta_{j,n}\|<r_{j,n}$, we have
\[
h_{\{\zeta_{j,n}\}}(z) \in \D(1,2^nr_{j,n}^\beta)\subset\D(1,\varepsilon_n)
\]
since $2^n r_{j,n}^\beta < \varepsilon_n$, so 
\[
\limsup_{w\rightarrow \partial \Omega\cap\mathbb{B}(\zeta_{j,n},r_{j,n})} u_{j,n}(w) \leq -\frac{1}{3}.
\]
Hence for each $\varepsilon>0$, we can construct a plurisubharmonic function of the form $u_\varepsilon(z) \coloneq \sum_{j,n \in \N} u_{j,n}(z)$ with $\limsup_{w\rightarrow A} u_\varepsilon(w) \leq -\frac{1}{3}$, which is not identically $-\infty$ as
\[
u_\varepsilon(a) > - \sum_{i,n \in \N} 2^n r_{j,n}^\beta \geq - \varepsilon.
\]
In particular, repeating the argument for $\varepsilon =\frac{1}{2},\frac{1}{2^2}, \frac{1}{2^3}, \dots$ we can define 
\[
u\coloneq\sum_k u_{\frac{1}{2^k}}
\]
with the properties that $\limsup_{w\rightarrow A}u(w) =-\infty$ and $u(a)>-\infty$; since $a$ was arbitrary, we conclude that $A$ is b-pluripolar and non-propagating.
\end{proof}
\begin{remark}
By modifying the above argument slightly, one could recover Theorem~\ref{oksendalsats} (and its accompanying anisotropic version for the strictly pseudoconvex case) from the initial constructions in the proofs of Davie--Øksendal (a peak function with a local Lipschitz or anisotropic estimate). One could also consider other moduli of continuity, given by a real valued, increasing function $h(r)$ with $\lim_{r\rightarrow0} h(r) = h(0) = 0$. One can to such  \textit{measuring functions} associate a $h$-Hausdorff measure $\mathcal{H}^h$, for which the sets of measure zero would be those that can be covered with balls of radii $r_j$ with $\sum h(r_j) < \varepsilon$. It would then be natural to at least require $h(r)\leq \log(\frac{1}{r})^{-N}$, because Labutin~\cite{labutin} proved that every $A \subset \C^N$ with $\mathcal{H}^{\log(\frac{1}{r})^{-N}}(A)<\infty$ must be pluripolar. He also proved that for every measuring function $h$ with
\[
\liminf_{r\rightarrow0^+}h(r)\log\Big(\frac{1}{r}\Big)^N=0,
\]
one can find a non-pluripolar set $A_h$ such that $\mathcal{H}^h(A_h)=0$. This illustrates that the situation for interior plurisubharmonic singularities is very different. 
\end{remark}
\begin{remark}
A construction due to Yu~\cite{yu} shows that there exists a smooth, B-regular domain $\Omega$ and a non-empty set $P \subset \partial \Omega$ such that no point in $P$ admits a peak function in $A(\Omega)$. It is currently unclear how large $P$ can be. 
\end{remark}
\section{Jensen measures do not coincide with representing measures}
Throughout this section, let $\Omega \Subset \C^N$ denote a smooth, strictly pseudoconvex domain and let $M(\partial \Omega)$ denote the space of finite Borel measures on $\partial \Omega$.  We let $M_{z_0}(\partial \Omega)\subset M(\partial\Omega)$ denote the set of \textit{representing measures} at a point $z_0 \in \Omega$, i.e., $\mu \in M(\partial\Omega)$  such that
\[
f(z_0) =\int_{\partial \Omega}f \ d\mu, \quad f\in A(\Omega).
\]
If $\mu$ satisfies
\[
\log |f(z_0)| \leq \int_{\partial \Omega} \log|f| \ d\mu, \quad f\in A(\Omega),
\]
we write $\mu \in J_{z_0}(\partial\Omega)$ and say that $\mu$ is a \textit{Jensen measure} at $z_0$. Our goal in this section is to generalize 
\begin{theorem}[Hedenmalm~\cite{hedenmalm}]\label{hedenmalmsats}
Let $\mathbb{B}$ denote the unit ball in $\C^N$, $N>1$, and let $\bar 0 \in \mathbb{B}$ denote its center. Then the set $M_{\bar 0}(\partial \mathbb{B})\setminus J_{\bar 0}(\partial \mathbb{B})$ is non-empty.
\end{theorem}
Hedenmalm's proof is by construction. Our strategy will be to non-constructively show that there exists a measure $\nu  \in M_{z_0}(\partial\Omega)$ and a compact set $K\subset \partial \Omega$ such that $\nu(K)>0$, with $\mu(K)=0$ for every $\mu \in J_{z_0}(\partial\Omega)$. Following~\cite{Rudin}, we will use the following terminology:

\begin{definition}
Given a family of measures $\mathcal{M}\subset M(\partial\Omega)$, we say that a set $A\subset \partial \Omega$ is \textit{totally null} with respect to $\mathcal{M}$ if $\mu(A)=0$ for every $\mu \in \mathcal{M}$.
\end{definition}
It is well-known (using the Cole--Range theorem and Bishop's theorem, cf. for example~\cite[pp.~204--205]{Rudin}) that a compact set $K$ is totally null with respect to $M_{z_0}(\partial\Omega)$ if and only if $K$ is a \textit{zero set} for $A(\Omega)$, in other words there exists $f\in A(\Omega)$ such that $f(\zeta) = 0$ for every $\zeta \in K$ and $f(z) \neq 0$ for every $z \in \widebar \Omega \setminus K$. We can then construct
\[
u(z) \coloneq \log|f(z)| - \max_{w\in \widebar \Omega} \log|f(w)|
\]
which shows that $K$ is b-pluripolar and non-propagating. In particular, for any compact, propagating b-pluripolar set $\tilde K$ that does not propagate to $z_0$ (take for example $\tilde K \coloneq \partial \Omega \cap \Pi$, where $\Pi$ is a complex hyperplane not containing $z_0$), there exists $\nu \in M_{z_0}(\partial \Omega)$ such that $\nu (\tilde K) >0$. To generalize Theorem~\ref{hedenmalmsats}, it is therefore enough to prove
\begin{theorem}\label{jensen}
Let $A\subset \partial \Omega$ be a b-pluripolar set that does not propagate to $z_0$. Then $A$ is totally null with respect to $J_{z_0}(\partial \Omega)$.
\end{theorem}
\begin{proof}
Let $\psh^b(\Omega)\coloneq \psh(\Omega) \cap \usc(\widebar \Omega)$ denote the upper bounded psh functions on $\Omega$ (with upper semicontinuous extension to the boundary). We shall first show that for every $\mu \in J_{z_0}(\partial\Omega)$, we have
\[
u(z_0) \leq \int_{\partial \Omega} u \ d\mu, \quad u\in \psh^b(\Omega).
\]
By Wikström's approximation theorem and the monotone convergence theorem, it is enough to consider $u \in \psh(\Omega) \cap C(\widebar \Omega)$. As in the proof of Lemma~\ref{approx}, we can find $\Omega_j\Supset \Omega$ and $u_{j} \in \psh(\Omega_j)\cap C(\Omega_j)$ such that $|u-u_{j}|<\frac{1}{j}$ on $\widebar \Omega$. Furthermore, as $\Omega$ is strictly pseudoconvex, we can assume that $\Omega_j$ is strictly pseudoconvex as well. Reasoning as in the proof of \cite[Corollary~1.3.10]{stoutbok}, the approximation result of Bremermann--Sibony yields $f_1,\dots,f_n\in \mathcal{O}(\Omega_j) \subset A(\Omega)$ and constants $1\geq c_1,\dots, c_n>0$ such that
\[
u_j -  \frac{1}{j}\leq \max_{k=1,\dots,n} c_k \log |f_k| \leq u_j
\]
on $\widebar \Omega$. Consequently, there is a $f_{k_0}$ such that
\[
u_j(z_0) -   \frac{1}{j}  \leq  c_{k_0} \log |f_{k_0}(z_0)| \leq u_j(z_0).
\]
Hence,
\[
u_j(z_0) - \frac{1}{j}\leq c_{k_0}\log|f_{k_0}(z_0)|\leq \int_{\partial \Omega}\max_{k=1,\dots,n} c_k \log |f_k| \ d\mu \leq\int_{\partial \Omega} u_j \ d\mu,
\]
and so the assertion  follows by letting $j\rightarrow \infty$. In the special case that $u_A$ is a negative psh function with $u_A(z_0)>-\infty$ and $u_A(\zeta) =-\infty$ for every $\zeta \in A$, then
\[
 -\infty< \int_{\partial \Omega }u_A \ d\mu\leq \int_Au_A \ d\mu.
\]
We conclude that $\mu(A)=0$, which finishes the proof.
\end{proof}
\begin{corollary}\label{hedenmalmgen} Let $\Omega\subset \C^N, N>1$ be a smooth, strictly pseudoconvex domain. Then for every $z_0 \in \Omega$, $M_{z_0}(\partial \Omega)\setminus J_{z_0}(\partial \Omega)$ is non-empty.
\end{corollary}
\begin{remark}
The assumption that $\Omega$ is smooth and strictly pseudoconvex can be relaxed somewhat, since the Cole--Range theorem holds on any domain where $A(\Omega)$ is \textit{tight}. See~\cite{boos} and the references therein for more details.
\end{remark}
\bibliographystyle{amsplain}
\bibliography{name}

\end{document}